\documentclass[11pt]{amsart}
\usepackage{amssymb,amsmath,latexsym,enumerate, dsfont}
\usepackage{graphicx,verbatim,enumerate,%
	ifpdf,mdwlist}
\usepackage{color}
\usepackage{mathabx}
\ifpdf
\usepackage{hyperref}
\else
\usepackage[hypertex]{hyperref}
\fi
\usepackage{graphicx}
\usepackage{cancel}
\usepackage{float}

\newcommand{\sequences}[0]{\ensuremath{2^{\mathbb{N}}}}
\newcommand{\words}[0]{\ensuremath{2^{<\mathbb{N}}}}
\newcommand{\N}[0]{\ensuremath{\mathbb{N}}}

\newcommand{\emptystr}{\epsilon}
\renewcommand{\P}{\mathbb{P}}
\newcommand{\Q}{\mathbb{Q}}
\newcommand{\concat}{\ensuremath{^\frown}}

\newtheorem{theorem}{Theorem}
\newtheorem{lemma}[theorem]{Lemma}

\theoremstyle{definition}
\newtheorem{definition}[theorem]{Definition}
\newtheorem{remark}[theorem]{Remark}

\begin{document}
	\title{Probabilistic vs deterministic gamblers}
	
	\author{Laurent Bienvenu}
	\address{LaBRI, CNRS \& Universit\'e de Bordeaux, France }
	\email{laurent.bienvenu@computability.fr}
	
	\author{Valentino Delle Rose}
	\address{Dipartimento di Ingegneria Informatica e Scienze Matematiche, Universit\`a Degli Studi di Siena, I-53100 Siena, Italy}
	\email{valentin.dellerose@student.unisi.it}
	
	\author{Tomasz Steifer}
	\address{Institute of Fundamental Technological Research, Polish Academy of Sciences, ul. Pawinskiego 5B, Warszawa, Poland}
	\email{tsteifer@ippt.pan.pl}
	\date{}
	
	
	\maketitle
	
	\begin{abstract}
		Can a probabilistic gambler get arbitrarily rich when all deterministic gamblers fail? We study this problem in the context of algorithmic randomness, introducing a new notion---almost everywhere computable randomness. A binary sequence $X$ is a.e.\ computably random if there is no probabilistic computable strategy which is total and succeeds on $X$ for positive measure of oracles. Using the fireworks technique we construct a sequence which is partial computably random but not a.e.\ computably random. We also prove the separation between a.e.\ computable randomness and partial computable randomness, which happens exactly in the uniformly almost everywhere dominating Turing degrees.
	\end{abstract}
	
	\section{Introduction}
	
	What does it mean for an infinite binary sequence~$X$ to be random? This may seem like a strange question at first since in classical probability theory, any infinite binary sequence drawn at random (with respect to the uniform distribution) has probability~$0$ to occur. Yet, the theory of algorithmic randomness gives us a way answer it from a computability perspective: $X$ is random if it does not possess any property of measure~$0$ which can be computably tested. There are many ways to formalize this, and hence many possible definitions of random sequence. One of the main approaches is the so-called unpredictability paradigm. We may say that a sequence~$X$ is unpredictable if no computable gambling strategy (or martingale) betting on the values of the bits of~$X$ and being rewarded fairly for its predictions can become arbitrarily rich during the course of the (infinite) game. The main two notions of randomness derived from this point of view are computable randomness and partial computable randomness, depending on whether we allow total computable or partial computable martingales. But in either case, the martingales considered are deterministic. 
	
	In this paper, we ask: do we get a stronger notion of randomness if we ask that~$X$ defeats not just all deterministically computable martingales, but also all probabilistically computable martingales? Usually, in computability theory, allowing probabilistic computations does not make a difference. This is in large part due to the foundational result that if a set $A \subset \N$ (or function $f: \N \rightarrow \N$, etc.) can be obtained by a probabilistic computation with positive probability, then it can in fact be obtained via a deterministic computation~\cite{de2016computability}. Yet this result is not necessarily an obstacle here as for a given~$X$, different runs of the probabilistic algorithm are allowed to produce different martingales, as long as with positive probability, the martingale output by the probabilistic algorithm defeats~$X$. And indeed, the main result of our paper is that probabilistic martingales \emph{do} in fact perform better than deterministic ones!

	We should note that probabilistic martingales were already considered by Buss and Minnes~\cite{buss2013probabilistic}. However, the applicability of their results for our purpose is limited. In particular, they studied two cases: probabilistic martingales which are total almost surely and probabilistic martingales which may be partial but nevertheless almost surely succeed on a given sequence. It is fairly easy to show that these cases reduce to computable and partial computable martingales respectively. The results of this paper are different and require more involved proofs.

	\subsection{Notation}
	The set of all infinite binary sequences is denoted by $\sequences$, while the set of finite binary strings is $\words$. The truncation of $x$ to the first $n$ bits is $x\restriction n$, while length of a string $\sigma$ is written by $|\sigma|$. We write $\tau\prec x$ when $\tau$ is a prefix of some $x$ (which might be a sequence or a string). The empty string is denoted by $\emptystr$, the concatenation of two strings $\sigma$ and $\tau$ by $\sigma \concat \tau$. 
	We are working with the product topology on $\sequences$, i.e., the topology generated by cylinder sets $[\sigma]=\{X\in\sequences:\sigma\prec X\}$. This means that open sets are of the form $\bigcup_{\sigma \in A} [\sigma]$ where $A$ is any set of strings. When $A$ is computably enumerable (c.e.), the set $\bigcup_{\sigma \in A} [\sigma]$ is called \emph{effectively open}. In this topology, the clopen sets are exactly the finite unions of cylinders. 
	
	We further equip $\sequences$ with the uniform measure~$\mu$, which is the measure where each bit of the sequence is equal to $1/2$ independently of the values of other bits. Formally, $\mu$ is the unique probability measure on the $\sigma$-algebra generated by cylinders for which $\mu([\sigma])=2^{-|\sigma|}$ for all~$\sigma$.
	
	As is common in computability theory, we sometimes identify sequences and strings with subsets of $\N$ (via characteristic function of the set) or paths in the full infinite binary tree. In particular, we say that $\sigma$ is on the left of $\tau$ if $\sigma$ is lesser than $\tau$ with respect to the lexicographical order.
	
	\subsection{Algorithmic randomness}
	
	Algorithmic randomness' goal is to assign a meaning to the notion of individual random string or sequence. While for strings we cannot reasonably hope for a clear separation between random and non-random (instead we have a quantitative measure of randomness: Kolmogorov complexity), for infinite binary sequences one can get such a separation. There are in fact many possible definitions. The most important one is called Martin-L\"of randomness and is defined as follows. A set $\mathcal{N} \subset \sequences$ is called effectively null if for every $n$ one can cover it by an effectively open set of measure at most~$\leq 2^{-n}$, uniformly in~$n$.  
	
	\begin{definition}
		A sequence $X \in \sequences$ is called \emph{Martin-L\"of random} if it does not belong to any effectively null set. 
	\end{definition}
	
	Said otherwise, $X$ is Martin-L\"of random if for every sequence $(\mathcal{U}_n)$ of uniformly effectively open sets such that $\mu(\mathcal{U}_n) \leq 2^{-n}$ for all~$n$ (such a sequence is known as a \emph{Martin-L\"of test}), we have $X \notin \bigcap_n \mathcal{U}_n$.\\
	
	An effectively null set corresponds to an atypical (= measure~$0$) property which can in some sense be effectively tested and therefore, a Martin-L\"of random sequence is one that withstands all computable statistical tests. The reason Martin-L\"of's definition of randomness is considered to be the central one is that it is both well-behaved (Martin-L\"of random sequences possess most properties one would expect from `random' sequences, including computability-theoretic properties) and robust, in that one can naturally get to the same notion by different approaches. For example, if we denote by $K$ the prefix-free Kolmogorov complexity function (see for example~\cite{Nies2009}), then the Levin-Schnorr theorem states that a sequence~$X$ is Martin-L\"of random if and only if $K(X \restriction n) \geq n - d$ for some~$d$ and all~$n$. Informally, this means that Martin-L\"of random sequences are exactly the `incompressible' ones. 
	
	As discussed above there is, however, another natural paradigm to define randomness (seemingly different from atypicality): unpredictability. We want to say that a sequence~$X$ is random if its bits cannot be guessed with better-than-average accuracy. This is formalized via the notion of martingale. 
	
	\begin{definition}
		A function $d:\words\rightarrow \mathbb{R}^{> 0}$ is called a \emph{martingale} if for all $\sigma\in\words$:
		\[d(\sigma)=\frac{d(\sigma 0)+d(\sigma 1)}{2} \]
		A martingale $d$ succeeds on a sequence $X$ if
		\[\limsup_{n\to\infty} d(X\restriction n)=\infty .\]
	\end{definition}
	
	A martingale represents the outcome of a gambling strategy in a fair game where the gambler guesses bits one by one by betting some amount of money at each stage, doubling the stake if correct, losing the stake otherwise, debts not being allowed. The quantity $d(\sigma)$ represents the capital of the gambler after having seen $\sigma$. Usually in the literature martingales are allowed to take value~$0$ but not allowing it makes no difference for the definitions that follow and avoids some pathological cases later in the paper. 
	
	Armed with the notion of martingale, we can now formulate an important definition of ``randomness'', known as computable randomness. 
	
	\begin{definition}
		A sequence $X\in\sequences$ is called \emph{computably random} if no computable martingale succeeds on $X$.
	\end{definition}
	
	In the above definition, we consider only martingales that are total computable. We would also like to allow partial computable martingales, but since they are not total functions in general, they are not even martingales in the above sense. To remedy this, one can simply define a partial martingale as a function $d$ taking values in  $\mathbb{R}^{> 0}$ whose domain is contained in $\words$ and closed under the prefix relation (if $d(\sigma)$ is defined, $d(\tau)$ is defined for every prefix~$\tau$ of $\sigma$) and furthermore for every~$\sigma$, $d(\sigma0)$ is defined if and only if $d(\sigma1)$ is defined and in case both are defined, the fairness condition $d(\sigma)=(d(\sigma 0)+d(\sigma 1))/2$ applies. Finally, success is defined in the same way as for martingales: we say that $d$ succeeds on~$X$ if $d(X \restriction n)$ is defined for all~$n$ and $\limsup_{n\to\infty} d(X\restriction n)=\infty$. We can now get the following strengthening of computable randomness.  
	
	\begin{definition}
		A sequence $X\in\sequences$ is called \emph{partial computably random} if no partial computable martingale succeeds on $X$.
	\end{definition}
	
	It is well-known that partial computable randomness is strictly stronger than computable randomness, but nonetheless strictly weaker than Martin-L\"of randomness (see~\cite{Nies2009}). \\
	
	Computable randomness and partial computable randomness are pretty robust notions. For example, it makes no difference whether we define success as achieving unbounded capital or as having a capital that tends to infinity. 
	
	\begin{lemma}[folklore, see~\cite{downey_algorithmic_2010}]\label{savinglemma}
		For every total (resp.\ partial) computable martingale $d$ there exists a (resp.\ partial) computable martingale $d'$ such that $d$ and $d'$ succeed on exactly the same sequences and for every $A\in\sequences$ we have $\limsup_{n\to\infty}d(A\restriction n)=\infty$ iff $\lim_{n\to\infty}d'(A\restriction n)=\infty$. Moreover, an index for $d'$ can be found effectively from an index for~$d$. 
	\end{lemma}
	
	Another important fact is that instead of considering computable real-valued martingales, we can restrict ourselves to rational valued martingales that are computable as functions from $\words$ to $\mathbb{Q}$ (which we sometimes refer to as \emph{exactly computable} martingales).

	\begin{lemma}[Exact Computation lemma, see~\cite{mayordomo1994contributions}]
		For every total (resp.\ partial) computable martingale $d$, there exists a total (resp.\ partial) exactly computable martingale $d'$ such that $d'$ succeeds on every sequence on which $d$ succeeds. Moreover, an index for $d'$ can be effectively obtained from an index for $d$. 
	\end{lemma}

	\subsection{Probabilistic martingales}
	
	The above definitions assume computable martingales (partial or total) are deterministic. Our goal is to understand whether probabilistic martingales (i.e., obtained by a probabilistic algorithm) can do better. Usually, to capture the idea of probabilistic algorithm, one appeals to probabilistic models of computation, such as probabilistic Turing machines. However, from a computability-theoretic perspective, where relativization to an oracle is a bread-and-butter object of study, it is equivalent to assume that an infinite sequence of random bits is drawn in advance and given as oracle to a deterministic Turing machine which then uses it as a source of randomness. Thus, we will consider \emph{partial computable oracle martingales}, that is, Turing functionals $d$ where for every oracle~$Y$, $d^Y$ (the function computed by the functional with~$Y$ given as oracle) is a partial martingale.  
	
	\begin{definition}
		A sequence $X\in\sequences$ is called \emph{a.e.~computably random} if for every partial computable oracle martingale $d$ the set of oracles $Y$ such that $d^Y$ is a total martingale and succeeds on $X$ has measure zero, i.e. \[\mu\left(\left\{Y\in\sequences: d^Y \ \text{is total and} \ \limsup_{n\to\infty} d^Y(X\restriction n)=\infty\right\}\right)=0.\] 
		$X$ is said to be \emph{a.e.~partial computably random} if for every partial computable oracle martingale $d$ the set of oracles $Y$ such that $d^Y$ succeeds on $X$ has measure zero. 
	\end{definition}
	
	Note that we could have equivalently defined a.e.~(partial) computably randomness directly from the relativization of (partial) computable randomness: a sequence~$X$ is a.e.~(partial) computably random if for almost every~$Y$, $X$ is (partial) computably random relative to~$Y$. \\
	
	The informal question `do probabilistic gamblers perform better than deterministic ones' can now be fully formalized by the following two questions: 
	
	\begin{itemize}
		\item Is a.e.~computable randomness equal to computable randomness?
		\item Is a.e.~partial computable randomness equal to partial computable randomness?
	\end{itemize}
	
	\label{bussminnes} In~\cite{buss2013probabilistic}, Buss and Minnes studied a restricted version of this problem. They considered a model of probabilistic martingales where one further requires $d^Y(\sigma)$ to be defined for all $\sigma$ and almost all~$Y$. This is a strong restriction which allows one to use an averaging technique. If $d$ is a probabilistic martingale with this property, it is easy to prove that the average $D$ defined by $D(\sigma) = \int_Y d^Y(\sigma)$ is a computable martingale. If $X$ is computably random, $D$ fails against~$X$, that is, there is a constant~$c$ such that $D(X \restriction n) < c$ for all~$n$. Moreover, by Fatou's lemma:
	\[
	\int_Y \liminf_n d^Y(X \restriction n) \leq \liminf_n D(X \restriction n) < c
	\]
	which in turn implies that the set $\{Y:  \liminf_n d^Y(X \restriction n) = \infty\}$ has measure~$0$. In other words, the set of $Y$ such that $d^Y$ strongly succeeds against~$X$ has measure~$0$. By Lemma~\ref{savinglemma}, this means that if a sequence $X$ is computably random if and only if for every probabilistic martingale with the Buss-Minnes condition, $d$ fails on~$X$ with probability~$1$. \\
	
	Our main result is that, in the general case, we no longer have an equivalence of the two models: probabilistic martingales are indeed stronger than deterministic ones. 
	
	\begin{theorem}\label{thm:main-result}
		There exist a sequence~$X$ which is partial computably random but not a.e.\ partial computably random and indeed not even a.e.\ computably random.  
	\end{theorem}
	
	
	We will devote the next sections to proving Theorem~\ref{thm:main-result}, but let us say a few words on why we believe it to be an interesting result. First of all, it is in stark contrast with Buss and Minnes' result that probabilistic martingales do not do any better than deterministic ones when they are required to be total with probability~$1$: in the general case, probabilistic martingales do better! Second, this is to our knowledge the first result of this kind in algorithmic randomness. If we were to define a.e.~Martin-L\"of randomness following the same idea (i.e., saying that $X$ is a.e.~Martin-L\"of random if for almost all~$Y$, $X$ is Martin-L\"of random relative to oracle $Y$), we would not get anything new, because a.e.~Martin-L\"of randomness coincides with Martin-L\"of randomness. This is a direct consequence of the famous van Lambalgen theorem~\cite{vanLambalgen1987}, which states that for every $A,B \in \sequences$, the join $A \oplus B = A(0)B(0)A(1)B(1) \ldots$ is Martin-L\"of random if and only if $A$ is Martin-L\"of random and $B$ is Martin-L\"of relative to $A$, if and only if $B$ is Martin-L\"of random and $A$ is Martin-L\"of random relative to~$B$. Now, let $X$ be Martin-L\"of random. For almost all~$Y$, $Y$ is Martin-L\"of random relative to~$X$ (this is simply the fact that the set of Martin-L\"of random sequences has measure~$1$, relativized to~$X$), thus $X \oplus Y$ is Martin-L\"of random, and thus $X$ is Martin-L\"of random relative to~$Y$. This shows that $X$ is a.e.~Martin-L\"of random. We see that van Lambalgen's theorem is key in this argument (we use it three times!). It was already known that the analogue of van Lambalgen for computable randomness fails~\cite{Yu2007}, but Theorem~\ref{thm:main-result} shows that it fails in a very strong sense. 
	
	Let us also remark that van Lambalgen's theorem shows that Martin-L\"of randomness implies a.e.\ (partial) computable randomness: if $X$ is Martin-L\"of random, it is also Martin-L\"of random relative to~$Y$ for almost every~$Y$, and thus also (partial) computably random relative to~$Y$ for almost every~$Y$. 
	
	\section{Turing degrees of a.e.~computably random sequences}
	Before moving to the proof of Theorem \ref{thm:main-result}, we give a simple degree-theoretic proof of a weaker result, namely a separation between computable randomness and a.e.~computable randomness. 
	
	Recall that every Martin-L\"of random sequence is computably random but a computable random sequence is not necessarily Martin-L\"of random.This separation has some interesting connections with classical computability theory, as witnessed by the following theorem (recall that a sequence $Y$ has \emph{high Turing degree}, or simply \emph{is high} if it computes some function $F: \N \rightarrow \N$ such that for every total computable function~$f$, $f(n) \leq F(n)$ for almost all~$n$). 
	
	\begin{theorem}[Nies, Stephan, Terwijn~\cite{NiesST2005}]\label{thm:nst}
		Let $Y \in \sequences$. If $Y$ computes a sequence~$X$ such that $X$ is computably random but not Martin-L\"of random, then $Y$ has high Turing degree. Conversely, if $Y$ has high Turing degree, then it computes some $X$ which is computably random but not Martin-L\"of random. 
	\end{theorem}
	
	It turns out that one can get an exact analogue of this theorem for a.e.\ computable randomness by replacing highness with a stronger notion: almost everywhere domination. A sequence $Y$ is said to have \emph{almost everywhere  dominating Turing degree, or a.e.\ dominating Turing degree} if it computes an almost everywhere dominating function $F$, that is, a function~$F$ such that for every Turing functional~$\Gamma$ and almost every~$Z$, if $\Gamma^Z$ is total, then $\Gamma^Z(n) \leq F(n)$ for almost all~$n$. See~\cite{Nies2009} for a more complete presentation of the history of this notion, originally due to Dobrinen and Simpson~\cite{DobrinenS2004}.
	
	\begin{theorem}\label{thm:degree-sep}
		Let $Y \in \sequences$. If $Y$ computes a sequence~$X$ such that $X$ is a.e.\ computably random but not Martin-L\"of random, then $Y$ has a.e.\ dominating Turing degree. Conversely, if $Y$ has a.e.\ dominating Turing degree, then it computes some $X$ which is a.e.\ computably random but not Martin-L\"of random. 
	\end{theorem}
	
	\begin{remark}
		Nies et al.'s theorem actually states a little more than what we wrote above, namely that the sequence~$X$ in the second part of the theorem can be chosen to be Turing equivalent to~$Y$. The analogue theorem is also true for a.e.\ computable randomness and a.e.\ domination but the proof becomes substantially more technical (we would need to introduce techniques to encode information into a computably random sequence) for only a small gain. 
	\end{remark}
	

	\begin{proof}
		Let us prove the first part of the theorem by its contrapositive. Let $X \in \sequences$ whose degree is not almost everywhere dominating. Suppose also $X$ is not Martin-L\"of random, i.e., $X \in \bigcap_n \mathcal{U}_n$ for $(\mathcal{U}_n)_{n \in \N}$ a sequence of uniformly effectively open sets with $\mu(\mathcal{U}_n) \leq 2^{-n}$. Consider the function $t^X$ defined by $t^X(n):= \min \{ s \mid  X \in \mathcal{U}_n[s] \}$. Since $X$ does not have a.e. dominating degree, there must exist a functional~$\Gamma$ such that 
		\[
		\mu \{Z \mid \Gamma^Z \text{is total and~} \exists^\infty n\  \Gamma^Z(n) > t^X(n)\} > 0 
		\]
		When $\Gamma^Z$ is total and $\Gamma^Z(n) > t^X(n)$ for infinitely many~$n$, we have $X \in \mathcal{U}_n[\Gamma^Z(n)]$ for infinitely many~$n$. Note that in that case $\mathcal{U}_n[\Gamma^Z(n)]$ is a clopen set which $Z$-uniformly computable in~$Z$. It is well-known that this type of test characterizes Schnorr randomness (a notion we will no discuss here but suffices to say that Schnorr randomness is weaker than computable randomness): a sequence $X$ is Schnorr random if and only if for every computable sequence of clopen sets $\mathcal{D}_n$ such that $\mu(\mathcal{D}_n) \leq 2^{-n}$, $X$ belongs to only finitely $\mathcal{D}_n$ (see for example~\cite[Lemma 1.5.9]{Bienvenu2008-PhD}). Relativized to~$Z$, this fact shows that $X$ is not $Z$-Schnorr random for a positive measure of~$Z$'s, thus not $Z$-computably random for a positive measure of~$Z$'s. \\
		
		The strategy to prove the second part of the theorem is to take the function~$F$ computed by~$Y$ and use it as a time bound on oracle martingales in order to `totalize'  them, which then allows us to use the averaging argument presented on page~\pageref{bussminnes}. In order for this to work, we must first prove that~$F$ can be assumed to be `simple' (in terms of Kolmogorov complexity). 
		
		\begin{lemma}\label{lem:simple-aed}
			If $Y$ has a.e.\ dominating Turing degree, it computes an a.e.\ dominating function $F$ such that $K(F(n))= O(\log n)$. 
		\end{lemma}

		\begin{proof}
			Let $(\Phi_i)_{i \in \N}$ be an enumeration of all Turing functionals and consider the universal functional~$\Psi$ where $\Psi^{0^i1A}=\Phi^A_i$. It is easy to see that a function $F$ is almost everywhere dominating if for almost all~$Z$, either $\Psi^Z$ is not total or $\Phi^Z(n) \leq F(n)$ for almost every~$n$. For each~$Z$, let $t^Z(n)$ be the minimum~$t$, if it exists, such that $\Phi^Z(k)$ converges in time $\leq t$ for all~$k \leq n$ and let $f^Z(n) = t^Z(n) + \max_{k \leq n} \Phi^Z(k)$. 
			
			Let $Y$ be of a.e.\ dominating degree and $F \leq_T Y$ an almost everywhere dominating function. 
			
			For each~$n$, let
			\[
			\mathcal{U}_n = \{Z \mid f^Z(n) \downarrow< \infty\}
			\]
			which is $\Sigma^0_1$ uniformly in~$n$. We can write
			\[
			\mathcal{U}_n = \bigcup_k \mathcal{U}_{n,k}
			\]
			where
			\[
			\mathcal{U}_{n,k} = \{Z \mid f^Z(n) \downarrow<k\}
			\]
			and note that $\mathcal{U}_{n,k}$ is a clopen set, computable uniformly in~$n,k$. 
			
			Since~$F$ is almost everywhere dominating, we have that for almost all~$Z$ and almost all~$n$, either $f^Z(n)$ is undefined or $f^Z(n) \leq F(n)$. Said otherwise, the set
			\[
			\mathcal{N}_0 = \limsup (\mathcal{U}_n \setminus \mathcal{U}_{n,F(n)}) 
			\]
			is a nullset. \\
			
			Now, for all~$n$, let $a_n \in [0,n^2]$ be the largest integer that $\mu(\mathcal{U}_{n,F(n)}) \geq a_n /n^2$ and $F'(n)$ be the smallest~$k$ such that  $\mu(\mathcal{U}_{n,k}) \geq a_n /n^2$. We see that $F'(n)$ is computable from~$F$ and furthermore, 
			\[
			K(F'(n)) \leq K(a_n) +O(1) \leq 2 \log(n^2) +O(1) \leq 4 \log n +O(1)
			\]
			By definition, we have $\mu(\mathcal{U}_{n,F(n)}) \setminus \mathcal{U}_{n,F'(n)}) \leq 1/n^2$. By the Borel-Cantelli lemma, 
			\[
			\mathcal{N}_1 = \limsup (\mathcal{U}_{n,F(n)} \setminus \mathcal{U}_{n,F'(n)}) 
			\]
			is a nullset. Thus, $\mathcal{N}_0 \cup \mathcal{N}_1$ is a nullset, which means that 
			\[
			\limsup (\mathcal{U}_n \setminus \mathcal{U}_{n,F'(n)}) 
			\]
			is also a nullset, which in turn means that for almost all~$Z$, for almost all~$n$, if $f^Z(n)$ is defined, then $f^Z(n) \leq F'(n)$. By definition of $f$, a fortiori, for almost all~$Z$, if $\Phi^Z$ is total, then $\Phi^Z(n) \leq F'(n)$ for almost all~$n$. Thus the function $F$' 
			\begin{itemize}
				\item is almost everywhere dominating
				\item is computable in~$F$, hence computable in~$Y$
				\item satisfies $K(F'(n)) = O(\log n)$
			\end{itemize} 
			
			which finishes the proof of the lemma.
		\end{proof}
		
		As alluded to above, the function~$F$ is going to be used as a time bound. To see what we mean by this, consider a total (not necessarily computable) non-decreasing function $\psi: \N \rightarrow \N$. Let $d$ be a (partial) exactly computable martingale. The time-bounded version of~$d$ with time bound~$\psi$ is the martingale $d^\psi$ which mimics~$d$ but only allows it a time $\psi(n)$ to compute its bets on strings of length~$n$. If $d$ has not made a decision by this stage (either because it is in fact undefined, or because the time of computation is greater than $\psi(n)$)), the casino exclaims \textit{``End of bets, nothing goes on the table!"} and the martingale is assumed to have placed an empty bet. Formally, $d^\psi(\emptystr)=d(\emptystr)$ and for any string~$\sigma$ and $b \in \{0,1\}$:
		
		\[
		d^\psi(\sigma b)= \left\{ \begin{array}{ll} d^\psi(\sigma) \cdot d(\sigma b)/d(\sigma) & \text{if both}~ d(\sigma0)[\psi(n+1)] \downarrow ~\text{and}~ d(\sigma1)[\psi(n+1)] \downarrow \\ d^\psi(\sigma)~\text{otherwise} \end{array} \right. 
		\] 
		
		By definition $d^\psi$ is always total, and when $d$ is total, if the bound $\psi$ dominates the convergence time of $d$ (that is, for almost all~$\sigma$, $d(\sigma)[\psi(|\sigma|)] \downarrow$), then $d^\psi$ and $d$ are within a multiplicative constant of one another, which in particular implies that $d^\psi$ succeeds on the same sequences as~$d$.  
		
		Now, let $(d_i)$ be the effective enumeration of all exactly computable martingales with oracle. Without loss of generality, assume that $d_i$ has a delay~$i$ imposed on it. Let $F$ be the a.e\ dominating function as above. Let $\hat{d}$ be the oracle martingale defined by
		\[
		\hat{d}^Z(\sigma) = \sum_i 2^{-i} d^{Z, F}_i(\sigma)
		\]
		($d^{Z, F}_i$ is the time-bounded version of~$d^Z_i$ with time bound~$F$). 
		
		It is a total martingale for every~$Z$ as all $d^{Z, F}_i$ are total martingales. Thus, its average~$D$ defined by
		\[
		D(\sigma) = \int_Z \hat{d}^Z(\sigma)
		\]
		is also a martingale. 
		
		Moreover, $D$ is $F$- (exactly)computable. Indeed, because of the time bound~$F$, the value of $d^{Z, F}_i(\sigma)$ only depends of the first $F(|\sigma|)$ bits of~$Z$, and because of the delay on~the $d_i$, only the martingales $(d_i)_{i \leq |\sigma|}$ matter in the computation of $D(\sigma)$. Thus the integral $\int_Z \hat{d}^Z(\sigma)$ is in fact a finite sum, can be computed from~$F(|\sigma|)$, hence the $F$-computability of~$D$. Even more precisely, the set of values $\{D(\sigma) \mid |\sigma| \leq n\}$ is computable from $F(n)$, and thus the Kolmogorov complexity of this set is at most $K(F(n)) +O(1) = O(\log n)$.   
		
		Let then~$X$ be the sequence which diagonalizes against~$D$ (the reader not familiar with this concept will find all the necessary definitions in the next section). Computing the first~$n$ bits of $X$ only requires to know the set of values $\{D(\sigma) \mid |\sigma| \leq n\}$. Thus, we have established:
		\begin{itemize}
			\item $X \leq_T F$
			\item $K(X \restriction n) \leq K(F(n)) +O(1) = O(\log n)$. 
		\end{itemize}
		
		Since $D$ does not succeed on~$X$, by the exact same calculation as page~\pageref{bussminnes}, for almost all~$Z$, $\hat{d}^Z$ does not succeed on~$X$, and thus $d^{Z,F}_i$ does not succeed on~$X$ for any~$i$.
		
		But we also know, since~$F$ is a.e.\ dominating, for all $i$, for almost every~$Z$, either $d^Z_i$ is partial, or $d^Z_i$ is total and its computation time is dominated by~$F$, hence $d^Z$ is within a multiplicative constant of $d^{Z,F}$. 
		
		Putting the two together, this entails that for almost all~$i$ and almost all~$Z$, either $d^Z_i$ is partial or it is total and does not succeed on~$X$. In other words, $X$ is a.e.\ computably random. 
		
		$X$ has therefore all the desired properties:
		\begin{itemize}
			\item It is a.e.\ computably random, 
			\item It is computable in~$F$ and thus computable in~$Y$,
			\item $K(X \restriction n) = O(\log n)$, ensuring that~$X$ is not only not Martin-L\"of random, but not even partial computably random using a result of Merkle~\cite{merkle2008complexity} (no partial computably random sequence can be of logarithmic complexity). 
		\end{itemize}
	\end{proof}
	
	An important result of Binns et al.~\cite{BinnsKLS2006} is that a.e.\ domination is strictly stronger than highness. Thus this gives us the promised weaker version of Theorem~\ref{thm:main-result}: there exists a sequence~$X$ which is computably random but not a.e.\ computably random. Indeed take a high Turing degree~$\mathbf{a}$ which is not a.e.\ dominating. By Theorem~\ref{thm:nst}, there is an~$X$ in~$\mathbf{a}$ which is computably random but not Martin-L\"of random hence not a.e.\ computably random by Theorem~\ref{thm:degree-sep}.

	\section{The main construction}
	
	We now turn to the full proof of Theorem~\ref{thm:main-result}. We first recall the standard method to build a partial computably random sequence (see for example~\cite{Nies2009}). Next, we combine this construction with the so-called `fireworks' technique which can be viewed as a probabilistic forcing to see how to defeat, with probabilistic martingales, sequences that have been built using this construction.

	
	
	\subsection{Defeating finitely many martingales}
	
	Let us begin by explaining how to construct a partial computably random sequence. Let us first consider the simple case where we are trying to defeat a single martingale~$d$, which we assume for the moment to be total computable, by making sure its capital does not go above a certain threshold. Up to multiplying~$d$ by a small rational, we may assume that that $d(\emptystr)<1$. By induction, suppose we have already built $X \restriction n$ in a way that $d(X \restriction i)<1$ for all $i \leq n$. By the fairness condition, either $d((X \restriction n)\concat 0) < 1$ or  $d((X \restriction n) \concat 1) < 1$. If the former is true, we set $X \restriction (n+1) = (X \restriction n)\concat 0$, otherwise we set $X \restriction (n+1) = (X \restriction n)\concat 1$.  Continuing in this fashion we ensure that the martingale $d$ does not succeed against~$X$ as its never reaches 2. Observe that when the martingale~$d$ is exactly computable, the sequence~$X$ is computable (uniformly in a code for~$d$). 
	
	Suppose now that we have a finite family of total martingales $d_1, \ldots d_n$. If we want to diagonalize against all of them at the same time, one can simply find positive rationals $q_1, \ldots, q_n$ such that $\sum_{i=1}^n q_i \cdot d_i(\emptystr) < 1$ and proceed as before against the martingale $\sum_{i=1}^n q_i \cdot d_i$. Again, the sequence~$X$ obtained by diagonalization against this finite family of martingales is computable uniformly in a code for the family of $d_i$'s. 
	But suppose now that some of the martingales in this family are partial instead of total. This does not cause much difficulty: having already built $X \restriction n$, consider only the sub-family $F$ of indices of martingales that are still defined on $(X \restriction n)\concat 0$ and $(X \restriction n)\concat 1$. The other martingales are undefined and thus will not succeed by fiat on the sequence~$X$. Now, if $\sum_{i \in F} q_i \cdot d_i((X \restriction n)\concat 0) <1$, set $X \restriction (n+1) = (X \restriction n) \concat 0$, otherwise set $X \restriction (n+1) = (X \restriction n)\concat 1$. Once again the sequence~$X$ defeats all of the $d_i$'s, some of them because they become undefined at some stage, some of them because their capital never exceeds $1/q_i$. Moreover, $X$ is still a computable sequence. It is not however computable uniformly in a code for the family of $d_i$'s because one needs to specify which martingales become undefined in the construction and when (this is a finite amount of information but it cannot be uniformly computed) but this is not an obstacle for our purposes. 
	
	To summarize these preliminary considerations, we can make the following definition. 
	
	\begin{definition}
		Let $(d_1,q_1), \ldots (d_n,q_n)$ be a finite family where each $d_i$ is a (code for) a partial computable martingale and $q_i$ a positive rational. Let $\sigma \in \words$ such that, calling $F$ the family of indices~$i$ such that $d_i(\sigma)$ converges, we have $\sum_{i \in F} q_i \cdot d_i(\sigma) <1$. Consider the computable sequence~$X$ defined inductively by $X \restriction |\sigma| = \sigma$ and if $X \restriction n$ is already built, letting $F_n$ be the family of indices such that $d_i((X \restriction n)\concat 0)$ converges, then $X \restriction (n+1) = (X \restriction n) \concat 0$ if $\sum_{i \in F_n} q_i \cdot d_i(X \restriction n)\concat 0) <1$ and $X \restriction (n+1) = (X \restriction n) \concat 1$ otherwise. This sequence is called the \emph{diagonalization against  $(d_1,q_1), \ldots, (d_n,q_n)$ above $\sigma$}. 
	\end{definition}
	
	\subsection{Defeating all partial computable martingales}
	
	When we have a countable family of martingales to diagonalize against, the standard way to proceed is to introduce them one by one during the game so that at any step we only have to diagonalize against a finite family as above. The delays between the introduction of martingales is flexible and therefore will be a parameter of the construction. \\
	
	\noindent \textbf{The diagonalizing sequence $\Delta((t_e)_{e \in N})$}. \\
	
	Let $(d_i)_{i \in \N}$ be a standard enumeration of partial computable rational valued martingales. Let $(t_e)_{e \in \N}$ be a family of integers. The sequence $\Delta((t_e)_{e \in N})$ is constructed by finite extension as follows. Start with the empty string $\sigma_0 = \emptystr$ and recursively do the following. Having built $\sigma_n$, let $q_{n+1}$ be a rational such that $\sum_{i \in F} q_i \cdot d_i(\sigma_n) <1$ where $F$ is the set of indices $i \in [1,n+1]$ such that $d_i(\sigma_n)$ converges. Let $Z$ be the diagonalization against $(d_1,q_1), \ldots, (d_{n+1},q_{n+1})$ above $\sigma_n$. The sequence $Z$ is an extension of $\sigma$ and is computable (see above), so let $e$ be a code for it (say the smallest one). Define $\sigma_{n+1}=Z \restriction ({|\sigma_n|+t_e})$. Finally, set
	\[
	\Delta((t_e)_{e \in N}) = \bigcup_n \sigma_n
	\]
	It is easy to check that $\Delta((t_e)_{e \in N})$ defeats all partial computable martingales. Moreover, the construction ensures the following important fact, which will be key for the rest of our proof:

	\textit{Fact 1:} For infinitely many~$e$ (namely, those codes that show up in the construction), the sequence $\Delta((t_e)_{e \in N})$ coincides with the computable sequence $Z$ of index~$e$ on a prefix of length~$\geq t_e$.

	\subsection{Fireworks}
	
	Let $(\P, \leq)$ be a computable order, that is, each element $p \in \P$ can be encoded by an integer and for a given pair $(n,m)$ of integers, it is decidable whether $n$ and $m$ are indeed codes for two elements of $p$ and $q$ in $\P$ and whether $p \leq q$. We say that a sequence $(p_i)_{i \in \N}$ of elements of $\P$ is \emph{$\P$-generic} if $p_0 \geq p_1 \geq p_2 \geq \ldots$ and for every c.e.\ subset $W$ of $\P$:
	\begin{itemize}
		\item either there exists an~$i$ such that $p_i \in W$
		\item or, there exists a~$j$ such that for any $q \leq p_j$, $q \notin W$
	\end{itemize}
	
	In particular, if $W$ is dense (that is, for every $p \in \P$ there exists $q \leq p$ such that $q \in W$), then for every generic sequence $(p_i)_{i \in \N}$ there must be some $i$ such that $p_i \in W$, in which case we say that $\P$ \textit{meets} $W$. 
	
	For most computable orders of interest, there cannot exist a computable generic sequence. However, there is a way to probabilistically obtain one, using the so-called fireworks technique. It was first invented by Kurtz~\cite{kurtz1982randomness} who showed that one can probabilistically obtain a generic sequence when $\P$ is the set of strings and $\sigma \leq \tau$ when $\tau$ is a prefix of $\sigma$. Rumyantsev and Shen~\cite{rumyantsev2014probabilistic} simplified Kurtz's presentation of this technique, which in turn allowed Bienvenu and Patey to make the following generalization to any computable order. 
	
	\begin{theorem}[Fireworks master theorem~\cite{bienvenu2017diagonally}]\label{thm:fireworks}
		For any computable order $\P$, there exists a Turing functional~$\Phi$ with range $\P$ such that for a set of $Z$'s of positive measure, we have that $\Phi^Z(i)$ is defined for all~$i$ and the sequence $(\Phi^Z(i))_{i \in \N}$ is generic. 
	\end{theorem}
	
	For our proof of Theorem~\ref{thm:main-result}, we are going to use the order $\P$ whose elements are finite approximations of martingales with positive rational values. Specifically, a member of $\P$ is a total function $f$ whose domain is $\{0,1\}^{\leq n}$ for some~$n$ -- which we call \emph{length of~$f$} and denote by $lh(f)$ -- whose range is $\Q^{>0}$, such that  $f(\emptystr)=1$ and $f(\sigma)=(f(\sigma0)+f(\sigma1))/2$ for all $\sigma$ of length $<lh(f)$. We say that $g \leq f$ if $g$ is an extension of $f$ (i.e., the domain of $f$ is contained in the domain of $g$ and the two coincide on the domain of~$f$). It is clear that $(\P,\leq)$ is a computable order. It is also clear that if $f_1 \geq f_2 \geq \ldots$ is a sequence of elements of $\P$ such that $lh(f_i)$ tends to $+\infty$, then $D=\bigcup f_i$ is a total rational valued martingale. This is in particular the case when $(f_i)_{i \in \N}$ is a $\P$-generic sequence, because for every $n$, the set of elements of $\P$ of length at least~$n$ is dense; in this case, we say that the martingale $D=\bigcup f_i$ is a \emph{$\P$-generic martingale}.

	\begin{lemma}\label{lem:delta-generic}
		Let $D$ be a $\P$-generic martingale. For every computable sequence~$Z$ and integer~$k$ there exists $s$ such that $D$ reaches capital at least $k$ while playing against the prefix of $Z$ of length~$s$ (that is, $D(Z \restriction l)>k$ for some $l<s$). 
	\end{lemma}

	\begin{proof}
		Fix a computable~$Z$ and consider the set
		\[
		W = \{g \in \P \mid (\exists l)~g(Z \restriction l)>k \}
		\]
		We claim that~$W$ is a dense c.e.\ subset of $\P$. That it is c.e.\ is clear. Now, take any $f \in \P$. Let $n=lh(f)$. By definition of $\P$, $f(Z \restriction n)$ is positive, so we can pick an $m>n$ such that $2^{m-n} \cdot f(Z \restriction n) > k$. Let $g$ be the martingale of length~$m$ which  behaves like~$f$ up to length~$n$ and after that stage plays the doubling strategy on~$Z$ (and stops betting outside of~$Z$). Formally:
		\[
		g(\tau) = \left \{  
		\begin{array}{ll}
		f(\tau) & ~ \text{if $|\tau| \leq n$}  \\
		f(\tau \restriction n) & ~ \text{if $|\tau| \geq n$ and $\tau \restriction n \not= Z \restriction n$} \\ 
		0 & ~ \text{if $\tau \restriction n = Z \restriction n$ but $\tau$ is not a prefix of~$Z$} \\
		f(Z \restriction n) \cdot 2^{|\tau|-n} & ~\text{if $\tau$ is a prefix of~$Z$}
		\end{array}
		\right. 
		\]
		
		It is easy to check that $g$ is a finite approximation of martingale which extends~$f$ and by construction $g(Z \restriction m) = 2^{m-n}\cdot f(Z \restriction n) > k$. Thus~$W$ is indeed dense. 
		
	\end{proof}
	
	We can now finish the proof of our main result. 
	
	\begin{proof}[Proof of Theorem~\ref{thm:main-result}]
		By Theorem~\ref{thm:fireworks} applied to our partial order $(\P,\leq)$, there is a Turing functional~$\Phi$ and a set $\mathcal{G}$ of positive measure such that for every $Z \in \mathcal{G}$, $\Phi^Z(n)$ is a $\P$-generic sequence. Thus for $Z \in \mathcal{G}$, $D^Z= \bigcup_n \Phi^Z(n)$ is a $\P$-generic martingale. \\
		
		Let $Z$ be a computable sequence and $e$ be a code for $Z$. By Lemma~\ref{lem:delta-generic}, for every $X \in \mathcal{G}$, there exists some $l^X_e$ such that $D^X$ -- being a $\P$-generic martingale -- reaches capital at least $e$ at some point while playing against the prefix $Z \restriction l^X_e$. 
		
		Now, for each~$e$ which is the code of a computable sequence choose some $s_e$ large enough to have 
		\[
		\mu \{Z \in \mathcal{G} \mid  l^X_e \leq s_e\} \geq (1- 2^{-e-1})\mu(\mathcal{G})
		\]
		(and for $e$ which is not a code for a computable sequence, choose $s_e$ arbitrarily). 
		
		This guarantees that 
		\[
		\mu \{Z \in \mathcal{G}  \mid (\forall e ~\text{code for a computable seq.})\,\,  l^X_e \leq s_e\} \geq \mu(\mathcal{G})/2 > 0
		\]
		Let $\mathcal{H}$ be the set of the left-hand side of this inequality. 
		
		Let us consider the sequence $\Delta((s_e)_{e \in \N})$, which by construction is partial computably random. For every $X \in \mathcal{H}$, for every computable sequence~$Z$ of code $e$, the martingale~$D^X$ reaches capital at least~$e$ on~$Z \restriction s_e$. On the other hand, by Fact 1, we know that for infinitely many~$e$, the sequence $\Delta((s_e)_{e \in N})$ coincides with the computable sequence $Z$ of index~$e$ on a prefix of length~$\geq s_e$. Thus this guarantees that for $X \in \mathcal{H}$, $D^X$ reaches capital at least~$e$ while playing on $\Delta((s_e)_{e \in \N})$. Thus $\Delta((s_e)_{e \in \N})$ is partial computably random but not almost everywhere computably random since~$\mathcal{H}$ has positive measure. 
	\end{proof}
	



	\newpage
	
	\bibliographystyle{plain}
	\bibliography{probabilistic_vs_deterministic_gamblers}

\begin{thebibliography}{10}

\bibitem{Bienvenu2008-PhD}
Laurent Bienvenu.
\newblock {\em Game-theoretic characterizations of randomness: unpredictability
  and stochasticity}.
\newblock PhD thesis, Universit{\'e} de Provence, 2008.

\bibitem{bienvenu2017diagonally}
Laurent Bienvenu and Ludovic Patey.
\newblock Diagonally non-computable functions and fireworks.
\newblock {\em Information and Computation}, 253:64--77, 2017.

\bibitem{BinnsKLS2006}
Stephen Binns, Bj{\o}rn Kjos-Hanssen, Manuel Lerman, and Reed Solomon.
\newblock On a conjecture of {D}obrinen and {S}impson concerning almost
  everywhere domination.
\newblock {\em Journal of Symbolic Logic}, 71(1):119--136, 2006.

\bibitem{buss2013probabilistic}
Sam Buss, Mia Minnes, et~al.
\newblock Probabilistic algorithmic randomness.
\newblock {\em Journal of Symbolic Logic}, 78(2):579--601, 2013.

\bibitem{de2016computability}
Karel de~Leeuw, Edward~F. Moore, Claude Shannon, and Norman Shapiro.
\newblock Computability by probabilistic machines.
\newblock In {\em Automata Studies}. Princeton University Press, 1956.

\bibitem{DobrinenS2004}
Natasha Dobrinen and Stephen~G. Simpson.
\newblock Almost everywhere domination.
\newblock {\em Journal of Symbolic Logic}, 69(3):914--922, 2004.

\bibitem{downey_algorithmic_2010}
Rodney~G. Downey and Denis~R. Hirschfeldt.
\newblock {\em Algorithmic {Randomness} and {Complexity}}.
\newblock Theory and {Applications} of {Computability}. Springer New York, New
  York, NY, 2010.

\bibitem{kurtz1982randomness}
Stuart~Alan Kurtz.
\newblock {\em Randomness and Genericity in the Degrees of Unsolvability.}
\newblock PhD thesis, University of Illinois at Urbana--Champaign,, 1982.

\bibitem{mayordomo1994contributions}
Elvira Mayordomo.
\newblock {\em Contributions to the study of resource-bounded measure}.
\newblock PhD thesis, Universitat Polit{\`e}cnica de Catalunya (UPC), 1994.

\bibitem{merkle2008complexity}
Wolfgang Merkle.
\newblock The complexity of stochastic sequences.
\newblock {\em Journal of Computer and System Sciences}, 74(3):350--357, 2008.

\bibitem{Nies2009}
Andr{\'e} Nies.
\newblock {\em Computability and randomness}.
\newblock Oxford Logic Guides. Oxford University Press, 2009.

\bibitem{NiesST2005}
Andr{\'e} Nies, Frank Stephan, and Sebastiaan Terwijn.
\newblock Randomness, relativization and {T}uring degrees.
\newblock {\em Journal of Symbolic Logic}, 70:515--535, 2005.

\bibitem{rumyantsev2014probabilistic}
Andrei Rumyantsev and Alexander Shen.
\newblock Probabilistic constructions of computable objects and a computable
  version of {L}ov{\'a}sz local lemma.
\newblock {\em Fundamenta Informaticae}, 132(1):1--14, 2014.

\bibitem{vanLambalgen1987}
Michiel van Lambalgen.
\newblock {\em Random sequences}.
\newblock Ph{D} dissertation, University of Amsterdam, Amsterdam, 1987.

\bibitem{Yu2007}
Liang Yu.
\newblock When van {L}ambalgen's theorem fails.
\newblock {\em Proceedings of the American Mathematical Society},
  135(3):861--864, 2007.

\end{thebibliography}

\end{document}